\documentclass[12pt]{amsart}
\usepackage{amsmath, amsthm, amssymb,cite}
\usepackage{fullpage}
\usepackage{color}
\usepackage[colorlinks=true]{hyperref}

\newtheorem{theorem}{Theorem}[section]
\newtheorem{lemma}{Lemma}

\numberwithin{lemma}{section}

\allowdisplaybreaks

\numberwithin{equation}{section}

\makeatletter
\renewcommand{\subsection}{%
  \@startsection{subsection}{2}{\z@}%
  {1.5ex plus .5ex minus .2ex}%
  {.8ex plus .2ex}%
  {\normalfont\bfseries}%
}
\makeatother

\author{Jianxiang Liu}
\address{Department of Mathematics \& Statistics, McMaster University}
\curraddr{}
\email{liu2238@mcmaster.ca}

\begin{document}

\title{Hessian Estimates for Convex Solutions of Hessian Quotient
Equations via Compactness}

\begin{abstract}
We give a compactness proof of interior Hessian estimates
for convex solutions of
$\sigma_n(D^2u)/\sigma_{n-2}(D^2u)=1$ and
$\sigma_n(D^2u)/\sigma_{n-1}(D^2u)=f(x,u,Du)>0$,
with $1/f$ concave in the gradient variable in the latter case,
revisiting results of Lu~\cite{Lu2025} and Fung~\cite{Fung2026}.
\end{abstract}

\maketitle

\tableofcontents

\section{Introduction}
\label{sec:1}

In this note, we give alternative compactness proofs of
interior a priori Hessian estimates previously established
by Lu~\cite{Lu2025} and Fung~\cite{Fung2026} for convex
solutions of two classes of Hessian quotient equations.
Throughout, $\sigma_k(D^2u)$ denotes the $k$-th elementary
symmetric polynomial of the eigenvalues of the Hessian $D^2u$,
with $\sigma_0=1$. We write
\begin{align}
\Gamma_k
:=\{\lambda\in\mathbb R^n:\sigma_j(\lambda)>0
\text{ for }j=1,\ldots,k\},
\end{align}
and write $A\in\Gamma_k$ when the eigenvalues of a symmetric
matrix $A$ belong to $\Gamma_k$.
We establish the following estimates.

\begin{theorem}\label{thm:interior-hessian}
Let $n\ge2$ and $R>0$, and suppose that $u$ is a smooth
convex solution of
\begin{align}\label{eq:mainpde}
\frac{\sigma_n(D^2u)}{\sigma_{n-2}(D^2u)}=1
\qquad\text{in }B_{2R}(0)\subset\mathbb R^n.
\end{align}
Then
\begin{align}
|D^2u(0)|
\le g\left(\frac{\|Du\|_{L^\infty(B_R(0))}}{R},n\right),
\end{align}
where $g(t,n)$ is a universal positive finite-valued function,
nondecreasing in $t$.
Consequently, by convexity,
\begin{align}
|D^2u(0)|
\le g\left(\frac{2\|u\|_{L^\infty(B_{2R}(0))}}{R^2},n\right).
\end{align}
\end{theorem}

For the quotient $\sigma_n/\sigma_{n-1}$, the compactness
argument also applies to a positive right-hand side
$f(x,u,Du)$, provided that $1/f$ is concave in the gradient
variable.

\begin{theorem}\label{thm:variable-rhs}
Let $n\ge2$ and $R>0$, and suppose that $u$ is a smooth
convex solution of
\begin{align}
\frac{\sigma_n(D^2u)}{\sigma_{n-1}(D^2u)}
=f(x,u,Du)
\qquad\text{in }B_R(0)\subset\mathbb R^n.
\end{align}
Suppose that
\begin{align}
\|u\|_{L^\infty(B_R(0))}\le M,
\qquad
\|Du\|_{L^\infty(B_R(0))}\le L,
\end{align}
and set
\begin{align}
K=\overline{B_R(0)}\times[-M,M]\times\overline{B_L(0)}.
\end{align}
Assume that $f=f(x,z,p)\in C^2(\Omega)$ for an open set
$\Omega\supset K$, with
\begin{align}
\inf_K f>0,
\qquad
D_{pp}^2\left(\frac1f\right)\le0
\qquad\text{on }K.
\end{align}
Then
\begin{align}
|D^2u(0)|
\le C\left(n,R,M,L,\|f\|_{C^2(K)},\inf_K f\right).
\end{align}
\end{theorem}

Interior Hessian estimates for the Poisson equation
$\Delta u=1$ are classical. However, Pogorelov's singular examples~\cite{Pogorelov1978}
show that such estimates fail in general for
$\det D^2u=\sigma_3(D^2u)=1$ in dimension three.
Urbas~\cite{Urbas1990} extended this phenomenon
to higher-order Hessian equations $\sigma_k(D^2u)=f$,
$n\ge k\ge3$.

These results highlight the special role of the quadratic Hessian equation in the theory of interior Hessian estimates. Heinz~\cite{Heinz1959} established interior $C^2$
estimates for the two-dimensional Monge--Amp\`ere equation. This result was obtained by an elementary pointwise proof by Chen--Han--Ou~\cite{ChenHanOu2016} and for a proof by Liu~\cite{Liu2021} using the partial Legendre transform.
In dimension three, Warren and Yuan~\cite{WarrenYuan2009,warrenYuan2010}
established interior Hessian estimates for
$\sigma_2(D^2u)=1$ by exploiting its special geometric
structure, while Qiu~\cite{Qiu2024hessian} treated the more general
right-hand side $f(x,u,Du)$. Zhou~\cite{Zhou2024} established interior Hessian estimates
for $\sigma_2(D^2u)=f^2(x)$ with positive Lipschitz $f$,
and extended the regularity to continuous viscosity solutions. In dimension four, Shankar and Yuan~\cite{ShankarYuan2025}
established estimates for the constant right-hand side,
and Fan~\cite{Fan2026D4} extended them to $f(x,u,Du)$.

In higher dimensions, various interior Hessian estimates
for the $\sigma_2$ equation under additional convexity
assumptions were established
in~\cite{GuanQiu2019,McGonagleSongYuan2019,
ShankarYuan2020,mooney2021,ShankarYuan2025,Fan2026D4}.
In recent year, Mooney~\cite{mooney2025} obtained interior Hessian estimates
under additional $W^{2,p}$ bounds for $p>2$,
while Chen, Jian, and Zhou~\cite{chenjianz2026}
treated convex solutions with positive Lipschitz right-hand sides.
More recently, Li and Wu~\cite{LiWu2026Sigma2}
established interior Hessian estimates for
$\sigma_2(D^2u)=1$ in arbitrary dimensions
for all admissible solutions, without any additional
convexity assumptions. 
As applications of interior Hessian estimates,
regularity results for viscosity solutions were obtained
in~\cite{shankaryuan2021,mooney2021,chenjiantuzhou2026,chenzhouzhu2026,zhouzhu2026}.

Meanwhile, progress has also been made on interior Hessian
estimates for Hessian quotient equations.
Early results for $\sigma_3(D^2u)/\sigma_1(D^2u)=1$
in dimensions three and four were obtained by exploiting
the special Lagrangian structure of the equation
\cite{ChenYuanWarren2009,WangYuan2014}.
Zhou~\cite{Zhou2024} treated the Lipschitz variable right-hand side
$f^2(x)$ using a twisted Lagrangian structure.
Subsequently, building on his earlier work in dimension three
for $f=f(x)$~\cite{Lu2023},
Lu~\cite{Lu2025} established interior Hessian estimates
for convex solutions of
$\sigma_n(D^2u)/\sigma_{n-i}(D^2u)=f(x,u)$,
$i=1,2$, in arbitrary dimensions, using the Legendre
transform and an integral argument.
Fung~\cite{Fung2026} extended these estimates to right-hand sides
$f(x,u,Du)$ with $1/f$ concave in the gradient variable,
using a doubling argument.
Following the approach of Urbas~\cite{Urbas1990},
Lu~\cite[Theorem 1.2]{Lu2025} also constructed counterexamples showing that such
estimates fail in general for $\sigma_k(D^2u)/\sigma_l(D^2u)=f$
when $1\le l<k\le n$ and $k-l\ge3$, even for convex solutions. And Fung~\cite[Remark 1.4]{Fung2026} extended this failure to $1\le l<k\le n$ if $f$ depends on $Du$ using the idea in~\cite[Theorem 1.2]{GuanRenWang2015}.

Subsequent research addressed the more general Hessian quotient
equations 
$$\sigma_k(D^2u)/\sigma_l(D^2u)=f(x,u), \quad l=k-1,k-2.$$
Lu and Tsai~\cite{lutsai2026note} established interior Hessian
estimates for convex solutions under an additional concavity
inequality, while Fung~\cite{Fung2026} extended the results
to gradient-dependent right-hand sides with a similar $1/f$ concavity to $Du$ and same assumption concavity inequality.
For $k=2$, Jiao and Sui~\cite{jiaosui2026}
established interior Hessian estimates via a doubling argument,
while Lu and Sroka~\cite{lusroka2026} obtained related estimates
for constant right-hand sides by reducing the equation to
the $\sigma_2$ equation, both under additional convexity assumptions.
Ji and Liang~\cite{jiliang2026} also obtained similar estimates
for semiconvex solutions with Lipschitz right-hand sides. For $k=3$, Mai and Yan also attained the estimates for semiconvex solutions by transferring the equation to a sum type and use the concave inequality by \cite{dongxuzhang2026}. More recently, Li and Wu~\cite{LiWu2026Quotient} and Dong and Zhang~\cite{DongZhang2026}
further developed the concavity inequalities in~\cite{tsai2026,guansroka2026} and established interior Hessian estimates for these equations with arbitrary $k\leq n$, for convex and admissible semiconvex solutions, respectively.
For constant right-hand sides, Wu~\cite{Wu2026} obtained
such estimates for $k$-convex solutions under $2k>n$,
while Qiu and Yan~\cite{QiuYan2026} established them for
admissible solutions in arbitrary dimensions.

All of the aforementioned results on Hessian quotient equations
rely on Jacobi inequalities and substantial computations.
Inspired by the compactness argument of McGonagle--Song--Yuan~\cite{McGonagleSongYuan2019} and Li~\cite{Licaiyan2019}, we develop a new
compactness proof of Theorems~\ref{thm:interior-hessian}
and~\ref{thm:variable-rhs}.
Our argument combines the strong minimum
principle~\cite{BardiDaLio1999,HanLin2011} with the area
formula~\cite{Simon1983} to obtain the final contradiction.
This approach avoids the derivation of a Jacobi inequality,
while the viscosity form of the strong minimum principle
allows us to treat gradient-dependent right-hand sides
under the stated concavity assumption.

\section{Proof of Theorem \ref{thm:interior-hessian}} \label{sec:2}

\begin{proof}
By the rescaling $\bar{u}(x)=u(Rx)/R^2$, it suffices to consider
the case $R=1$. For simplicity, we write $B_r=B_r(0)$ from now on.

\subsection*{Step 1. A contradiction sequence and its limit.}

Suppose, to the contrary, that the estimate fails.
Then there exist a constant $L>0$ and a sequence of smooth
convex solutions $\{u_k\}_{k=1}^{\infty}$ of equation \eqref{eq:mainpde} such that
\begin{align}
    \|Du_k\|_{L^\infty(B_1)} \le L,
    \qquad
    |D^2u_k(0)| \longrightarrow \infty
    \quad \text{as } k\to\infty.
\end{align}

We normalize each $u_k$ by setting
\begin{align}
    \widetilde{u}_k(x)
    =u_k(x)-u_k(0)-Du_k(0)\cdot x.
    \label{eq:normalization}
\end{align}
Since $D^2\widetilde{u}_k=D^2u_k$, this normalization preserves
both the equation and convexity. Moreover,
\begin{align}
    \widetilde{u}_k(0)=0,\qquad
    D\widetilde{u}_k(0)=0,\qquad
    \|D\widetilde{u}_k\|_{L^\infty(B_1)}\le 2L.
\end{align}
For simplicity, we continue to denote the normalized sequence
by $\{u_k\}$.

For all $x,x_1,x_2\in B_1$, we have
\begin{align}
    |u_k(x)|\le 2L,
    \qquad
    |u_k(x_1)-u_k(x_2)|\le 2L|x_1-x_2|.
\end{align}
Thus each $u_k$ extends continuously to $\overline{B}_1$,
with the same bounds. The sequence is therefore uniformly
bounded and equicontinuous on $\overline{B}_1$.
By the Arzel\`a--Ascoli theorem, after passing to a subsequence,
still denoted by $\{u_k\}$, we obtain
\begin{align}
    u_k\longrightarrow u_\infty
    \qquad \text{uniformly on } \overline{B}_1.
\end{align}

\subsection*{Step 2. The Legendre transform and the transformed equation.}\label{step:2}

Set $a=\sqrt{n(n-1)/2}\ge 1$ and define
\begin{align}
v_j(x)=u_j(x)+\frac{a}{2}|x|^2,
\qquad x\in\overline{B}_1,\quad j\in\mathbb{N}\cup\{\infty\}.
\end{align}
The uniform limit $u_\infty$ is convex.
Since $a\ge 1$, each $v_j$ is $1$-strongly convex on
$\overline{B}_1$. Moreover,
\begin{align}\label{eq:v-uniform-convergence}
\|v_k-v_\infty\|_{L^\infty(\overline{B}_1)}
=
\|u_k-u_\infty\|_{L^\infty(\overline{B}_1)}
\longrightarrow 0
\qquad\text{as }k\to\infty.
\end{align}

Define the Legendre transform
\begin{align}\label{eq:legendre-transform}
w_j(y)=\max_{x\in\overline{B}_1}
\{x\cdot y-v_j(x)\},
\qquad y\in\mathbb{R}^n,
\end{align}
and let $X_j(y)\in\overline{B}_1$ denote the unique maximizer.
Existence and uniqueness follow from continuity, compactness,
and strict convexity of $v_j$.

The functions $w_j$ are continuously differentiable on
$\mathbb{R}^n$, and
\begin{align}
Dw_j(y)&=X_j(y),
\label{eq:legendre-gradient}\\
|Dw_j(y_1)-Dw_j(y_2)|&\le |y_1-y_2|,
\qquad y_1,y_2\in\mathbb{R}^n.
\label{eq:legendre-lipschitz}
\end{align}
These are standard properties of the Legendre transform;
see, for example, \cite[Theorem 18.15, Corollary 16.30]{BauschkeCombettes2017}.
Moreover,
\begin{align}\label{eq:legendre-gradient-convergence}
\sup_{y\in\mathbb{R}^n}
|Dw_k(y)-Dw_\infty(y)|^2
\le 2\|v_k-v_\infty\|_{L^\infty(\overline{B}_1)}
\longrightarrow 0.
\end{align}
For completeness, a direct proof of (\ref{eq:legendre-gradient}--\ref{eq:legendre-gradient-convergence}) is provided in
Lemma~\ref{lem:legendre-properties} in the
\hyperref[sec:Appendix]{Appendix}.

We define the subdifferential of $v_\infty$ at the origin by
\begin{align}
\partial v_\infty(0)
:=
\left\{
p\in\mathbb{R}^n:
v_\infty(x)\ge v_\infty(0)+p\cdot x
\text{ for all }x\in\overline{B}_1
\right\}.
\end{align}
Since
\begin{align}
\|Dv_k\|_{L^\infty(B_1)}
=\|Du_k+ax\|_{L^\infty(B_1)}
\le 2L+a=:M,
\end{align}
continuous extension to $\overline{B}_1$ and uniform convergence
give
\begin{align}
|v_\infty(x_1)-v_\infty(x_2)|
\le M|x_1-x_2|,
\qquad x_1,x_2\in\overline{B}_1.
\end{align}
Thus $\partial v_\infty(0)$ is a compact convex subset of
$\overline{B}_M$.
Moreover, the normalization $Du_k(0)=0$ and convexity imply
$v_k(x)\ge v_k(0)$; passing to the limit yields
$0\in\partial v_\infty(0)$.

Set
\begin{align}\label{eq:definition-N}
N:=
\left\{
y\in\mathbb{R}^n:
\operatorname{dist}(y,\partial v_\infty(0))<\frac14
\right\}.
\end{align}
Then $N\subset B_{M+1/4}$ is open and convex, hence connected,
and contains the origin.\label{properties of N}

For $y\in\partial v_\infty(0)$, the defining inequality gives
\begin{align}
x\cdot y-v_\infty(x)
\le -v_\infty(0),
\qquad x\in\overline{B}_1.
\end{align}
Hence the unique maximizer is $X_\infty(y)=0$, and therefore
$Dw_\infty(y)=0$.
The converse follows from the same inequality, so
\begin{align}\label{eq:subdifferential-zero-gradient}
y\in\partial v_\infty(0)
\quad\Longleftrightarrow\quad
Dw_\infty(y)=0.
\end{align}

For every $y_1\in N$, choose $y_2\in\partial v_\infty(0)$
such that $|y_1-y_2|<1/4$. By \eqref{eq:legendre-lipschitz} and
\eqref{eq:subdifferential-zero-gradient},
\begin{align}
|Dw_\infty(y_1)|
=
|Dw_\infty(y_1)-Dw_\infty(y_2)|
\le |y_1-y_2|<\frac14.
\end{align}
Consequently, $Dw_\infty(N)\subset B_{1/4}$.
By uniform convergence~\eqref{eq:v-uniform-convergence}, choose $k_0\in\mathbb N$ such that
\begin{align}
\|v_k-v_\infty\|_{L^\infty(\overline{B}_1)}
<\frac1{32},
\qquad k>k_0.
\end{align}
Then \eqref{eq:legendre-gradient-convergence} gives
\begin{align}
\sup_{y\in\mathbb{R}^n}|Dw_k(y)-Dw_\infty(y)|<\frac14,
\end{align}
and hence
\begin{align}\label{eq:gradient-image-interior}
Dw_k(N)\subset B_{1/2},
\qquad k>k_0.
\end{align}

Therefore, for $k>k_0$ and $y\in N$, the maximizer
$X_k(y)=Dw_k(y)$ lies in the interior of $B_1$.
The first-order optimality condition therefore gives
\begin{align}\label{eq:inverse-gradient-identity}
y=Dv_k(X_k(y))=Dv_k(Dw_k(y)),
\qquad y\in N.
\end{align}
Since $v_k\in C^\infty(B_1)$ and $D^2v_k\ge I$,
the inverse function theorem implies that $Dw_k$ locally
coincides with a smooth inverse of $Dv_k$, so $w_k\in C^\infty(N)$.
Differentiating \eqref{eq:inverse-gradient-identity} yields
\begin{align}\label{eq:inverse-hessian-identity}
I=D^2v_k(Dw_k(y))\,D^2w_k(y),
\qquad y\in N,\quad k>k_0.
\end{align}

For $k>k_0$, convexity of $u_k$ and the original equation \eqref{eq:mainpde} imply
$D^2u_k>0$. By \eqref{eq:inverse-hessian-identity},
\begin{align}\label{eq:transformed-hessian}
D^2u_k(Dw_k(y))
=[D^2w_k(y)]^{-1}-aI,
\qquad
0<D^2w_k(y)<\frac1a I.
\end{align}
Hence, suppressing the evaluation points,
\begin{align}\label{eq:inverse-shift-identity}
(I-aD^2w_k)^{-1}-I=a(D^2u_k)^{-1}.
\end{align}
Since $\sigma_{n-2}(D^2u_k)/\sigma_{n}(D^2u_k)=\sigma_2((D^2u_k)^{-1})=1$ and
$a^2=n(n-1)/2$, we obtain
\begin{align}
a^2
&=\sigma_2\bigl((I-aD^2w_k)^{-1}-I\bigr)\notag\\
&=\sigma_2\bigl((I-aD^2w_k)^{-1}\bigr)
-(n-1)\sigma_1\bigl((I-aD^2w_k)^{-1}\bigr)+a^2.
\end{align}
Since $I-aD^2w_k>0$ by \eqref{eq:transformed-hessian}, this gives
\begin{align}\label{eq:transformed-quotient}
q(I-aD^2w_k)=\frac1{n-1},
\qquad
q=\frac{\sigma_{n-1}}{\sigma_{n-2}}.
\end{align}
Define $G(A)=-q(I-aA)$. Then
\begin{align}\label{eq:transformed-equation}
G(D^2w_k)=-\frac1{n-1}
\qquad\text{in }N.
\end{align}
Let $0<\lambda_1\le\cdots\le\lambda_n$ be the eigenvalues of
$D^2u_k$. Since $\sigma_2(\lambda^{-1})=1$, we have
$\lambda_2^{-2}\le\lambda_1^{-1}\lambda_2^{-1}\le1$,
so $\lambda_i\ge1$ for $i\ge2$. Thus the eigenvalues
$\eta_i=\lambda_i/(\lambda_i+a)$ of $I-aD^2w_k$ satisfy
\begin{align}\label{eq:transformed-eigenvalue-bounds}
0<\eta_1\le1,
\qquad
\frac1{1+a}\le\eta_i\le1,
\quad i=2,\ldots,n.
\end{align}
Consequently,
\begin{align}\label{eq:lowerbound of G}
\sigma_{n-2}(I-aD^2w_k(y))
\ge (1+a)^{-(n-2)}>0,
\qquad y\in N,\quad k>k_0.
\end{align}
Moreover, $\sigma_{n-2}(I-aA)>0$ on
\begin{align*}
\mathcal D
=\{A\in\operatorname{Sym}(n):
\lambda(I-aA)\in\Gamma_{n-1}\},
\end{align*}
so $G\in C^\infty(\mathcal D)$.

Writing $\eta_i=1-a\mu_i$, where $\mu_i$ are the
corresponding eigenvalues of $D^2w_k$, we now establish
the convexity and uniform ellipticity of
\eqref{eq:transformed-equation}.

\medskip
\textbf{Convexity:} The concavity of $q$ on $\Gamma_{n-1}$
(see, for example, \cite{Trudinger1995})
implies that $G$ is convex on $\mathcal D$.

\textbf{Uniform ellipticity:}
Under the rescaling $\bar\lambda=a\eta$, equation
\eqref{eq:transformed-quotient} and the range $0<\eta_i\le1$
correspond to those in \cite{ChangYuan2010} and \cite[Step~2]{McGonagleSongYuan2019},
with $K=\delta=a^{-1}$.
The ellipticity argument there gives
\begin{align}
c(n)\le q_i(\bar\lambda)\le C(n),
\qquad i=1,\ldots,n.
\end{align}
Since $q$ is homogeneous of degree one,
$Dq(a\eta)=Dq(\eta)$, so the same bounds hold for $q_i(\eta)$.
Thus, $G(\mu)=-q(1-a\mu)$ implies $G_i=a q_i(\eta)$ and after adjusting the constants,
\begin{align}\label{eq:elliptic_bound}
c(n)I\le \bigl(G^{ij}(D^2w_k(y))\bigr)\le C(n)I,
\qquad y\in N,\quad k>k_0.
\end{align}

\subsection*{Step 3. Compactness and the limiting equation.}\label{step:3}

Since $Dw_k=X_k\in\overline{B}_1$ and $w_k(0)=0$,
the convexity of $N$ and the inclusion $0\in N$ give
\begin{align}
|w_k(y)|
&=\left|\int_0^1 Dw_k(ty)\cdot y\,dt\right|
\le |y|
\le M+\frac14,
\qquad y\in N.
\end{align}
By the Evans--Krylov--Safonov theory, there exists
$\alpha=\alpha(n)\in(0,1)$ such that, for every
$U\Subset N$,
\begin{align}
\|w_k\|_{C^{2,\alpha}(U)}
\le C\bigl(n,M,\operatorname{dist}(U,\partial N)\bigr),
\qquad k>k_0.
\end{align}
For any fixed $\beta\in(0,\alpha)$, a diagonal argument
yields a subsequence, still denoted by $\{w_k\}$, such that
\begin{align}
w_k\longrightarrow\widetilde w
\qquad\text{in }C^{2,\beta}_{\mathrm{loc}}(N).
\end{align}

The definition of the Legendre transform gives
\begin{align}
|w_k(y)-w_\infty(y)|
\le \|v_k-v_\infty\|_{L^\infty(\overline{B}_1)}
\longrightarrow0,
\qquad y\in\mathbb R^n.
\end{align}
Thus $\widetilde w=w_\infty$ in $N$, and hence
\begin{align}\label{eq:legendre-C2-convergence}
w_k\longrightarrow w_\infty
\qquad\text{in }C^{2,\beta}_{\mathrm{loc}}(N).
\end{align}

Therefore, the bounds \eqref{eq:transformed-eigenvalue-bounds}
pass to the limit, ensuring that
$D^2w_\infty(y)\in\mathcal D$ for every $y\in N$. By the \hyperref[eq:lowerbound of G]{smoothness} of $G$ on $\mathcal D$ and the local $C^{2,\beta}$ convergence,
we obtain
\begin{align}\label{eq:limiting-equation}
G(D^2w_\infty)=-\frac1{n-1},
\qquad D^2w_\infty\ge0
\qquad\text{in }N.
\end{align}

We now show that $D^2w_\infty(0)$ is singular.
By the normalization, $Dv_k(0)=0$.
Together with $Dv_k(Dw_k(0))=0$ and the injectivity of
$Dv_k$, this implies $Dw_k(0)=0$.
The inverse Hessian identity therefore yields
\begin{align}
D^2w_k(0)
=\bigl(D^2v_k(0)\bigr)^{-1}
=\bigl(D^2u_k(0)+aI\bigr)^{-1}.
\end{align}
Since $D^2u_k(0)\ge0$ and $|D^2u_k(0)|\to\infty$,
we obtain
\begin{align}
\lambda_{\min}(D^2w_k(0))
=\frac1{\lambda_{\max}(D^2u_k(0))+a}
\longrightarrow0.
\end{align}
By \eqref{eq:legendre-C2-convergence},
\begin{align}
\lambda_{\min}(D^2w_\infty(0))=0.
\end{align}
After an orthogonal change of coordinates, we may assume that
\begin{align}\label{eq:limiting-degeneracy}
D_{11}w_\infty(0)=0.
\end{align}

\subsection*{Step 4. Degeneracy of the limiting Hessian.}

The local $C^{2,\beta}$ convergence \eqref{eq:legendre-C2-convergence} and the \hyperref[eq:lowerbound of G]{smoothness} of $G$ on $\mathcal D$ allow us to pass the ellipticity bounds \eqref{eq:elliptic_bound} from \hyperref[step:2]{Step~2} to the limit:
\begin{align}
c(n)I\le \bigl(G^{ij}(D^2w_\infty)\bigr)\le C(n)I
\qquad\text{in }N.
\end{align}
By the difference quotient method and interior Schauder estimates,
followed by bootstrapping, we obtain $w_\infty\in C^\infty(N)$.

Differentiating \eqref{eq:limiting-equation} with respect to $y_1$
gives
\begin{align}
\sum_{i,j}G^{ij}(w_\infty)_{ij1}=0.
\end{align}
Differentiating once more, we obtain
\begin{align}
\sum_{i,j}G^{ij}(w_\infty)_{ij11}
+\sum_{i,j,r,s}G^{ij,rs}
(w_\infty)_{ij1}(w_\infty)_{rs1}=0,
\end{align}
where the derivatives of $G$ are evaluated at $D^2w_\infty$. By the convexity of $G$,
\begin{align}
\sum_{i,j}G^{ij}\partial_{ij}\bigl((w_\infty)_{11}\bigr)
\le0
\qquad\text{in }N.
\end{align}
Since $(w_\infty)_{11}\ge0$ in $N$ and
$(w_\infty)_{11}(0)=0$, the strong maximum principle
and the \hyperref[properties of N]{connectedness} of $N$ yield
\begin{align}
(w_\infty)_{11}\equiv0
\qquad\text{in }N.
\end{align}
As $D^2w_\infty$ is positive semidefinite, it follows that
$D^2w_\infty e_1=0$, where $e_1=(1,0,\ldots,0)^T$.
Consequently,
\begin{align}
\det D^2w_\infty=0
\qquad\text{in }N.
\end{align}

By \eqref{eq:legendre-lipschitz}, $Dw_\infty$ is
globally Lipschitz. Since every point of $Dw_\infty(N)$
has at least one preimage in $N$, the area formula (see for example \cite[(8.2)]{Simon1983}) gives
\begin{align}
|Dw_\infty(N)|
&\le \int_{\mathbb R^n}
\#\{y\in N: Dw_\infty(y)=x\}\,dx \notag\\
&= \int_N |\det D^2w_\infty(y)|\,dy
=0.
\end{align}

We claim that there exists $\rho\in(0,1)$ such that
\begin{align}\label{eq:subgradient-inclusion}
\partial v_\infty(x)\subset N
\qquad\text{for every }x\in B_\rho.
\end{align}
Otherwise, Suppose that \eqref{eq:subgradient-inclusion} fails for every
$\rho\in(0,1)$. Then, for each integer $j\ge2$, there exist
$x_j\in B_{1/j}$ so $x_j\to0$ and $p_j\in\partial v_\infty(x_j)$
such that $p_j\notin N$. 

Since $v_\infty$ is $M$-Lipschitz on $\overline B_1$,
we have $|p_j|\le M$.
After passing to a subsequence, we may assume that $p_j\to p$.
By the definition of the subgradient,
\begin{align*}
v_\infty(x)\ge v_\infty(x_j)+p_j\cdot(x-x_j),
\qquad x\in\overline B_1.
\end{align*}
Passing to the limit gives
\begin{align*}
v_\infty(x)\ge v_\infty(0)+p\cdot x,
\qquad x\in\overline B_1.
\end{align*}
Thus $p\in\partial v_\infty(0)\subset N$.
Since $N$ is open, $p_j\in N$ for all sufficiently large $j$,
a contradiction. This proves \eqref{eq:subgradient-inclusion}.

For every $x\in B_\rho$, the convexity of $v_\infty$
ensures that $\partial v_\infty(x)$ is nonempty.
Choose $p\in\partial v_\infty(x)\subset N$.
By the definition of the Legendre transform and the uniqueness
of the maximizer, we have
\begin{align}
x=X_\infty(p)=Dw_\infty(p).
\end{align}
Consequently,
\begin{align}
B_\rho\subset Dw_\infty(N),
\end{align}
contradicting $|Dw_\infty(N)|=0$.
This completes the proof.

\end{proof}

\section{Proof of Theorem~\ref{thm:variable-rhs}} \label{sec:3}

\begin{proof}

We follow the compactness argument of Section~\ref{sec:2} and retain its
notation wherever applicable. 

Steps 1--3 follow those of Section~\ref{sec:2} with the modifications
described below. As before, we reduce to $R=1$ and suppose
that there exist smooth convex solutions $u_k$ of
\begin{align}
\frac{\sigma_n(D^2u_k)}{\sigma_{n-1}(D^2u_k)}
=g_k(x,u_k,Du_k)
\qquad\text{in }B_1,
\end{align}
with uniformly bounded $C^1$ norms, such that
$|D^2u_k(0)|\to\infty$.
The functions $g_k$ satisfy the uniform bounds and structural
condition in Theorem~\ref{thm:variable-rhs}. In particular,
\begin{align}
g_k\ge m>0,
\qquad D_{pp}^2(g_k^{-1})\le0.
\end{align}
Since
\begin{align}
\sum_{i=1}^n\lambda_i(D^2u_k)^{-1}
=\frac1{g_k(x,u_k,Du_k)}
\le\frac1m,
\end{align}
we have $D^2u_k\ge mI$.

Set
\begin{align}
v_k(x)=u_k(x)-u_k(0)-Du_k(0)\cdot x.
\end{align}
The compactness and Legendre-transform arguments of Section~\ref{sec:2}
apply with strong convexity constant $m$. In particular, the gradient Lipschitz constant becomes $m^{-1}$,
and we take
\begin{align}
N=\left\{
y:\operatorname{dist}(y,\partial v_\infty(0))<\frac m4
\right\}.
\end{align}
For all sufficiently large $k>k_0$, the Legendre transforms satisfy
$Dw_k(N)\subset B_{1/2}$ and
\begin{align}\label{eq:transformed equation2}
0<D^2w_k\le m^{-1}I,
\qquad \Delta w_k=h_k(Q_k(y))
\qquad\text{in }N,
\end{align}
where $h_k=g_k^{-1}$ and
\begin{align}
Q_k(y)=\bigl(
&Dw_k(y),\,
y\cdot Dw_k(y)-w_k(y)+u_k(0)
+Du_k(0)\cdot Dw_k(y),\,
y+Du_k(0)
\bigr).
\end{align}

The chain rule, the uniform $C^1$ bound $\|f\|_{C^1(K)}$ on $g_k$,
the lower bound $g_k\ge m$, and $D^2w_k\le m^{-1}I$ give
\begin{align}
\|D(h_k\circ Q_k)\|_{L^\infty(N)}\le \frac{C(n)}{m^2}
\left(1+\frac{1+L}{m}\right)
\|f\|_{C^1(K)},
\end{align}
with a bound independent of $k$.
Since $N$ is convex and $N\subset B_{2L+\frac{m}{2}}$,
for every $\alpha\in(0,1)$,
\begin{align}
\|h_k\circ Q_k\|_{C^{0,\alpha}(N)}
&\le \frac1m
+\|D(h_k\circ Q_k)\|_{L^\infty(N)}
  (\operatorname{diam}N)^{1-\alpha}
\le C.
\end{align}
Interior Schauder estimates therefore yield, for every
$U\Subset N$,
\begin{align}
\|w_k\|_{C^{2,\alpha}(U)}
&\le C\left(
\|w_k\|_{L^\infty(N)}
+\|h_k\circ Q_k\|_{C^{0,\alpha}(N)}
\right)
\le C,
\end{align}
where the constant also depends on
$\operatorname{dist}(U,\partial N)$ but is independent of $k$.
The compactness and limit-identification argument of Section~\ref{sec:2}
then gives, for $0<\beta<\alpha$,
\begin{align}\label{eq:limit-identification}
w_k\longrightarrow w_\infty
\qquad\text{in }C^{2,\beta}_{\mathrm{loc}}(N).
\end{align}
Moreover,
\begin{align}
D^2w_\infty\ge0,
\qquad \lambda_{\min}(D^2w_\infty(0))=0.
\end{align}

Step 4 requires a new argument, since the limiting right-hand side
of \eqref{eq:transformed equation2} may not be twice differentiable,
so we cannot directly differentiate the limiting equation twice. Inspired by Han-lin's\cite[Section 5.2]{HanLin2011}, we now establish the propagation of degeneracy in another way. Let
\begin{align}
\lambda_k(y)=\lambda_{\min}(D^2w_k(y)).
\end{align}
We claim that, with $C_1,C_2>0$ independent of $k$,
\begin{align}\label{eq:min-eigenvalue-inequality}
-\Delta\lambda_k+C_1|D\lambda_k|+C_2\lambda_k\ge0
\qquad\text{in }N
\end{align}
in the viscosity sense.

Fix $y_0\in N$ and a unit eigenvector $e$ corresponding to
$\lambda_k(y_0)$, and keep $e$ constant when differentiating.
Write $\lambda=\lambda_k(y_0)$ and $p=y_0+Du_k(0)$.
At $y_0$,
\begin{align}
\partial_eQ_k
&=(\lambda e,\lambda p\cdot e,e),\\
\partial_{ee}Q_k
&=\bigl(
D(w_k)_{ee},\,
\lambda+p\cdot D(w_k)_{ee},\,
0
\bigr).
\end{align}
Differentiating the transformed equation \eqref{eq:transformed equation2} twice gives
\begin{align}
\Delta(w_k)_{ee}
={}&\bigl(D_xh_k+(h_k)_z p\bigr)\cdot D(w_k)_{ee}
+(h_k)_z\lambda \notag\\
&+D^2h_k[\partial_eQ_k,\partial_eQ_k],
\end{align}
with the derivatives of $h_k$ evaluated at $Q_k(y_0)$.
Since $D_{pp}^2h_k\le0$, the uniform bounds and
$0\le\lambda\le m^{-1}$ imply
\begin{align}
D^2h_k[\partial_eQ_k,\partial_eQ_k]\le C\lambda.
\end{align}
Consequently,
\begin{align}
\Delta(w_k)_{ee}(y_0)
\le C_1|D(w_k)_{ee}(y_0)|+C_2\lambda_k(y_0).
\end{align}

Let $\varphi$ be any $C^2$ function defined near $y_0$
that touches $\lambda_k$ from below at $y_0$. Then
\begin{align}
\varphi\le\lambda_k\le(w_k)_{ee}
\end{align}
near $y_0$, with equality at $y_0$.
Thus the preceding inequality gives
\begin{align}
\Delta\varphi(y_0)
\le C_1|D\varphi(y_0)|+C_2\lambda_k(y_0),
\end{align}
which proves \eqref{eq:min-eigenvalue-inequality} in the viscosity sense.

The local $C^{2,\beta}$ convergence \eqref{eq:limit-identification} implies that
\begin{align}
\lambda_k\longrightarrow
\lambda_\infty:=\lambda_{\min}(D^2w_\infty)
\qquad\text{locally uniformly in }N.
\end{align}
By stability of viscosity supersolutions
\cite[Remark~6.2]{CrandallIshiiLions1992},
\begin{align}
H[\lambda_\infty]
:=-\Delta\lambda_\infty+C_1|D\lambda_\infty|
+C_2\lambda_\infty\ge0
\qquad\text{in }N.
\end{align}

We apply \cite[Corollary~1]{BardiDaLio1999} to
$-\lambda_\infty$ with
$F(r,p,A):=-H(-r,-p,-A)=-\operatorname{tr}A-C_1|p|+C_2r$.
The operator $F$ is continuous and proper, so $(H0)$ there holds.
Moreover, $F(\xi r,\xi p,\xi A)=\xi F(r,p,A)$ for $\xi>0$,
which verifies $(H1)$ with $\varphi(\xi)=\xi$.
For every $p\ne0$,
\begin{align}
F(0,p,I-\gamma p\otimes p)
=-n+\gamma|p|^2-C_1|p|>0
\qquad\text{for }\gamma>\frac{n+C_1|p|}{|p|^2}.
\end{align}
Thus $(H2)$ holds with $\rho_0=1$.
Since $-\lambda_\infty$ is a continuous viscosity subsolution
of $F=0$ in the connected open set $N$ and attains its
nonnegative maximum $0$ at the origin, Corollary~1 there yields
$\lambda_\infty\equiv0$ in $N$.
Hence $\det D^2w_\infty\equiv0$ in $N$, and the area formula
and subgradient arguments of Section~\ref{sec:2}
give the same contradiction.

\end{proof}

\section{Appendix}\label{sec:Appendix}

We retain the notation $v_j$, $w_j$, and $X_j$ introduced in \hyperref[step:2]{Step~2} of Section~\ref{sec:2} and establish some standard properties of the Legendre transform
in our setting.
\begin{lemma}\label{lem:legendre-properties}
For every $j\in\mathbb{N}\cup\{\infty\}$, the function $w_j$
is continuously differentiable on $\mathbb{R}^n$, with
$Dw_j=X_j$ and
\begin{align}\label{eq:legendre-lipschitz2}
|Dw_j(y_1)-Dw_j(y_2)|\le |y_1-y_2|,
\qquad y_1,y_2\in\mathbb{R}^n.
\end{align}
Moreover, as $k \to \infty$,
\begin{align}\label{eq:legendre-gradient-convergence2}
\sup_{y\in\mathbb{R}^n}
|Dw_k(y)-Dw_\infty(y)|^2
\le 2\|v_k-v_\infty\|_{L^\infty(\overline{B}_1)}
\longrightarrow 0.
\end{align}
\end{lemma}

\begin{proof}
Fix $j\in\mathbb{N}\cup\{\infty\}$ and $y_1\in\mathbb{R}^n$,
and set $x_1=X_j(y_1)$.
By the minimizing property of $x_1$ and the $1$-strong
convexity of $v_j$, for any $x_2\in\overline{B}_1$ and $0<t<1$,
\begin{align*}
v_j(x_1)-y_1\cdot x_1
&\le v_j((1-t)x_1+tx_2)
-y_1\cdot((1-t)x_1+tx_2)\\
&\le (1-t)(v_j(x_1)-y_1\cdot x_1)
+t(v_j(x_2)-y_1\cdot x_2)
-\frac12t(1-t)|x_1-x_2|^2.
\end{align*}
Rearranging and letting $t\to0$ gives
\begin{align}\label{eq:legendre-strong-support}
v_j(x_2)\ge v_j(x_1)+y_1\cdot(x_2-x_1)
+\frac12|x_1-x_2|^2.
\end{align}
Now take $y_2\in\mathbb{R}^n$ and set $x_2=X_j(y_2)$.
Applying \eqref{eq:legendre-strong-support} and then
interchanging the indices $1$ and $2$, we obtain
\begin{align*}
v_j(x_1)&\ge v_j(x_2)+y_2\cdot(x_1-x_2)
+\frac12|x_1-x_2|^2.
\end{align*}
Adding these inequalities yields
\begin{align*}
|x_1-x_2|^2\le (y_1-y_2)\cdot(x_1-x_2).
\end{align*}
Hence $|X_j(y_1)-X_j(y_2)|\le |y_1-y_2|$.

By the maximizing properties of $x_1$ and $x_2$ and the definition of Legendre transform,
\begin{align*}
x_1\cdot(y_2-y_1)
\le w_j(y_2)-w_j(y_1)
\le x_2\cdot(y_2-y_1).
\end{align*}
Consequently,
\begin{align*}
0
&\le w_j(y_2)-w_j(y_1)-X_j(y_1)\cdot(y_2-y_1)\\
&\le (X_j(y_2)-X_j(y_1))\cdot(y_2-y_1)
\le |y_2-y_1|^2.
\end{align*}
Letting $y_2\to y_1$ proves $Dw_j(y_1)=X_j(y_1)$.
Since $X_j$ is Lipschitz, $Dw_j$ is continuous, and
\eqref{eq:legendre-lipschitz2} follows.

Finally, fix $y\in\mathbb{R}^n$ and set
$x_1=X_k(y)$ and $x_2=X_\infty(y)$.
Applying \eqref{eq:legendre-strong-support} to $v_k$
and $v_\infty$, respectively, gives
\begin{align*}
v_k(x_2)&\ge v_k(x_1)+y\cdot(x_2-x_1)
+\frac12|x_1-x_2|^2,\\
v_\infty(x_1)&\ge v_\infty(x_2)+y\cdot(x_1-x_2)
+\frac12|x_1-x_2|^2.
\end{align*}
Adding these inequalities, we obtain
\begin{align*}
|x_1-x_2|^2
&\le (v_k-v_\infty)(x_2)
-(v_k-v_\infty)(x_1)\\
&\le 2\|v_k-v_\infty\|_{L^\infty(\overline{B}_1)}.
\end{align*}
Taking the supremum over $y\in\mathbb{R}^n$ and using
the uniform convergence $v_k\to v_\infty$ proves
\eqref{eq:legendre-gradient-convergence2}.

\end{proof}

\section*{Acknowledgements}
This work forms part of the author's ongoing doctoral research.
The author would like to thank Professor Siyuan Lu for helpful
discussions and for drawing his attention to related work.
He is also deeply grateful to Professor Yu Yuan for valuable
discussions and for sharing his unpublished notes with him.

\section*{Use of artificial intelligence}

The author used GPT-6 Astra to assist with
language editing, bibliographic search, and
reviewing and discussing mathematical proofs.
All AI-assisted suggestions were critically reviewed
by the author, who takes full responsibility for
the mathematical correctness and content of this paper.

\bibliography{main}
\bibliographystyle{amsplain}
\end{document}